\documentclass[11pt]{article}%
\usepackage{amsfonts}
\usepackage{amsmath}
\usepackage{amssymb}
\usepackage{graphicx}%
\providecommand{\U}[1]{\protect\rule{.1in}{.1in}}
\newtheorem{theorem}{Theorem}

\newtheorem{corollary}[theorem]{Corollary}

\newtheorem{definition}[theorem]{Definition}

\newtheorem{lemma}[theorem]{Lemma}

\newtheorem{proposition}[theorem]{Proposition}
\newtheorem{remark}[theorem]{Remark}

\newenvironment{proof}[1][Proof]{\noindent\textbf{#1.} }{\ \rule{0.5em}{0.5em}}
\begin{document}

\title{Born--Jordan Pseudodifferential Operators and the Dirac Correspondence: Beyond
the Groenewold--van Hove Theorem}
\author{Maurice A. de Gosson\thanks{maurice.de.gosson@univie.ac.at}\\University of Vienna\\Faculty of Mathematics, NuHAG\\Oskar-Morgenstern-Platz 1, 1090 Vienna, Austria
\and Fabio Nicola\thanks{fabio.nicola@polito.it}\\Dipartimento di Scienze Matematiche\\Politecnico di Torino, \\Corso Duca degli Abruzzi 24, 10129 Torino, Italy}
\maketitle

\begin{abstract}
Quantization procedures play an essential role in microlocal analysis,
time-frequency analysis and, of course, in quantum mechanics. Roughly speaking
the basic idea, due to Dirac, is to associate to any symbol, or observable,
$a(x,\xi)$ an operator $\mathrm{Op}(a)$, according to some axioms dictated by
physical considerations. This led to the introduction of a variety of
quantizations. They all agree when the symbol $a(x,\xi)=f(x)$ depends only on
$x$ or $a(x,\xi)=g(\xi)$ depends only on $\xi$:
\[
\mathrm{Op}(f\otimes1)u=fu,\quad\mathrm{Op}(1\otimes g)u=\mathcal{F}%
^{-1}(g\mathcal{F}u)
\]
where $\mathcal{F}$ stands for the Fourier transform. Now, Dirac aimed at
finding a quantization satisfying, in addition, the key correspondence
\[
\lbrack\mathrm{Op}(a),\operatorname*{Op}(b)]=i\mathrm{Op}(\{a,b\})
\]
where $[\,,\,]$ stands for the commutator and $\{\,,\}$ for the Poisson
brackets, which would represent a tight link between classical and quantum
mechanics. Unfortunately, the famous Groenewold--van Hove theorem states that
such a quantization does not exist, and indeed most quantization rules satisfy
this property only approximately.

Now, in this note we show that the above commutator rule in fact holds for the
Born-Jordan quantization, at least for symbols of the type $f(x)+g(\xi)$.
Moreover we will prove that, remarkably, this property completely
characterizes this quantization rule, making it the quantization which best
fits the Dirac dream.


\end{abstract}

\section{Introduction}

The theory of pseudodifferential operators has had many avatars since its
inception in the mid 1960s; it has developed into a major branch of operator
theory since the pioneering work of R. Beals, H. Duistermaat, C. Fefferman, L.
H\"{o}rmander, J. J. Kohn, R. Melrose, L. Nirenberg, M. A. Shubin, M. E.
Taylor, and many others. One early precursor, having its origin in quantum
mechanics, and which gained its mathematical \textit{lettres de noblesse
}only\textit{ } in 1979 following the work of H\"{o}rmander \cite{horweyl}, is
the theory of Weyl operators. It was observed by Stein \cite{Stein},
\S 75--7.6, that the Weyl pseudodifferential calculus is \textit{uniquely
characterized} by its symplectic covariance with respect to conjugation with
metaplectic operators (see Wong \cite{wong} for a detailed proof of this
property). In the present paper we consider another class of
pseudodifferential operators, which is a relative newcomer in the mathematical
literature, and which we are going to characterize in terms of the so-called
Dirac correspondence. These operators are the Born--Jordan pseudodifferential
operators familiar to mathematicians working in quantization problems and in
time-frequency analysis. They can be defined as follows (we will give
alternative definitions as we go): assuming here for simplicity that
$a\in\mathcal{S}(\mathbb{R}^{2n})$ we define for, $\tau\in\mathbb{R}$, the
Shubin operator $\operatorname*{Op}_{\tau}(a)$ by%
\[
\operatorname*{Op}\nolimits_{\tau}(a)u(x)=(2\pi)^{-n}\int e^{i\langle
x-y,\xi\rangle}a((1-\tau)x+\tau y)u(y)dyd\xi
\]
for $u\in\mathcal{S}(\mathbb{R}^{n})$ (the case $\tau=\frac{1}{2}$
corresponding to Weyl operators).

The Born--Jordan operator $\operatorname*{Op}\nolimits_{\mathrm{BJ}}(a)$ with
symbol $a$ is then, by definition, the average%
\[
\operatorname*{Op}\nolimits_{\mathrm{BJ}}(a)=\int_{0}^{1}\operatorname*{Op}%
\nolimits_{\tau}(a)d\tau.
\]
Now, for a given quantization rule $a\longmapsto\operatorname*{Op}(a), $
regarded as a mapping
\[
\mathcal{S}^{\prime}(\mathbb{R}^{2n})\longrightarrow\mathcal{L(S}%
(\mathbb{R}^{n}),\mathcal{S}^{\prime}(\mathbb{R}^{n})),
\]
natural desirable properties are the formulas \eqref{one} and \eqref{two} below:%

\begin{equation}
\operatorname*{Op}(f\otimes1)u=fu,\quad\operatorname*{Op}(1\otimes
f)u=\mathcal{F}^{-1}(f\mathcal{F}u) \label{one}%
\end{equation}
for $u\in\mathcal{S}(\mathbb{R}^{n})$, where $\mathcal{F}$ stands for the
Fourier transform in $\mathbb{R}^{n}$, and
\begin{equation}
\lbrack\operatorname*{Op}(a),\operatorname*{Op}(b)]=i\operatorname*{Op}%
(\{a,b\}) \label{two}%
\end{equation}
where $[\cdot,\cdot]$ is the commutator and $\{a,b\}$ the Poisson bracket
associated with the standard symplectic form.

While property (\ref{one}) is a natural requirement for any honest
pseudodifferential calculus, property (\ref{two}) (which is closely related to
the physicists \textquotedblleft Dirac correspondence\textquotedblright) is of
a slightly more subtle nature. For a better understanding of the importance of
this property we have to briefly recall the notion of prequantization (see
Gotay \cite{62}, Gotay et al. \cite{61}, and Tuynman \cite{tuy} for detailed
discussions of the state of the art; also see Berndt \cite{berndt}; Englis
\cite{englis} discusses in a short paper the existence of nonlinear
quantizations and Abraham Marsden \cite{AM} address the question from a more
function-theoretical point of view). One requires that if $a\longmapsto
\operatorname*{Op}(a)$ is a continuous linear mapping associating to a real
symbol $a$ on $\mathbb{R}^{2n}$ a symmetric operator $\operatorname*{Op}(a)$
defined on some dense subspace of $L^{2}(\mathbb{R}^{n})$ then this mapping
should satisfy, in addition to some other axioms, the condition \eqref{two}.
It turns out that it is in principle impossible to achieve this goal; this
impossibility is the famous result of Groenewold \cite{groenewold}, later
completed by van Hove \cite{vanhove1,vanhove2}, which is a \textquotedblleft
no-go\textquotedblright\ result. It says (in its strong form) that one cannot
quantize the Poisson algebra of polynomials in $\mathbb{R}^{n}$, beyond those
of degree $\leq2$ (we briefly discuss this at the end of the paper).

In the present note we will show that:

\begin{itemize}
\item {} \textit{This obstruction can be bypassed if one limits oneself to
symbols of the type
\begin{equation}
a(x,\xi)=f(x)+g(\xi)\label{afg}%
\end{equation}
and choose $\operatorname*{Op}(a)=\operatorname*{Op}\nolimits_{\mathrm{BJ}%
}(a)$};

\item \textit{Conversely, $\operatorname*{Op}\nolimits_{\mathrm{BJ}}$ is the
only pseudodifferential quantization satisfying (\ref{two}) at least for
symbols of the type (\ref{afg}).}
\end{itemize}

In \eqref{afg} the functions $f,g$ are smooth and are allowed to grow at most
polynomially, together with their derivatives. Notice that for the Weyl
quantization $\mathrm{Op}_{1/2}$ the formula \eqref{two} holds, in general,
only for polynomial symbols of order $\leq2$.

Observe that assuming conditions \eqref{one} and \eqref{two}, at least for
symbols of the type \eqref{afg}, uniquely forces the values $\mathrm{Op}(a)$
when $a$ is in the linear space spanned of symbols of type \eqref{afg} and
their Poisson brackets. We will show that this space is dense in
$\mathcal{S}^{\prime}(\mathbb{R}^{2n})$ so that $\mathrm{Op}(a)$ is then
uniquely characterized by these properties; in fact, we have
$\operatorname*{Op}(a)=\operatorname*{Op}\nolimits_{\mathrm{BJ}}(a)$ for every
$a\in\mathcal{S}^{\prime}(\mathbb{R}^{2n})$.\par\smallskip

The importance of these results is double. First, the theory Born--Jordan
operators has recently gained considerable interest under the impetus of
mathematicians working in harmonic analysis \cite{13,13bis,26,transam} and
mathematical physicists \cite{Springer,golu,Kauffmann,73}. Secondly, as we
anticipated, it is intimately related to a mathematical question harking back
to the work of Groenewold \cite{groenewold} and van Hove
\cite{vanhove1,vanhove2} on quantization; we will discuss this at the end of
the paper.

This work is structured as follows: we review in Section \ref{sec1} the basic
properties of Born and Jordan's pseudodifferential calculus we will need to
prove our main results (Theorems \ref{Theorem1} and \ref{Theorem2}) in Section
\ref{sec2}. In Section \ref{secdisc} we discuss our results from the point of
view of quantization.

\subsubsection*{Notation}

We identify $\mathbb{R}^{n}$ with its dual $(\mathbb{R}^{n})^{\ast}$ and
$T^{\ast}\mathbb{R}^{n}$ with $\mathbb{R}^{2n}$; if $x\in\mathbb{R}^{n}$ and
$\xi\in\mathbb{R}^{n}$ we sometimes write $z=(x,\xi)$. The Schwartz space of
rapidly decreasing functions on $\mathbb{R}^{n}$ is denoted by $\mathcal{S}%
(\mathbb{R}^{n})$ and its dual (the tempered distributions) by $\mathcal{S}%
^{\prime}(\mathbb{R}^{n})$. We denote by $\delta_{(x^{\prime},\xi^{\prime})}$
the Dirac distribution centered at $(x^{\prime},\xi^{\prime})$, and by
$\mathcal{L(S}(\mathbb{R}^{n}),\mathcal{S}^{\prime}(\mathbb{R}^{n}))$ the
space of all continuous linear operators from $\mathcal{S}(\mathbb{R}^{n})$ to
$\mathcal{S}^{\prime}(\mathbb{R}^{n})$ (the continuity being understood in the
weak sense). The Euclidean scalar product of $x\in\mathbb{R}^{n}$ and $\xi
\in\mathbb{R}^{n}$ will be written $\langle x,\xi\rangle$. For $x=(x_{1}%
,\ldots,x_{n})$ we write $D=D_{x}=(D_{x_{1}},\ldots,D_{x_{n}})$ where
$D_{x_{j}}=-i\partial_{x_{j}}$ and $\langle x,D\rangle=x_{1}D_{x_{1}}%
+\cdot\cdot\cdot+x_{n}D_{x_{n}}$. The Fourier transform $\widehat{u}$ of
$u\in\mathcal{S}(\mathbb{R}^{n})$ is the function $\widehat{u}\in
\mathcal{S}(\mathbb{R}^{n})$ defined by%
\[
\widehat{u}(\xi)=\mathcal{F}u(\xi)=(2\pi)^{-n/2}\int e^{-i\langle\xi,x\rangle
}u(x)\,dx
\]
where $dx=dx_{1}\cdot\cdot\cdot dx_{n}$ is the usual Lebesgue measure on
$\mathbb{R}^{n}$. The standard symplectic form on $T^{\ast}\mathbb{R}%
^{n}=\mathbb{R}^{2n}$ is defined by $\sigma=d\xi_{1}\wedge dx_{1}+\cdot
\cdot\cdot+d\xi_{n}\wedge dx_{n}$; in coordinates
\[
\sigma(x,\xi;y,\eta)=\langle\xi,y\rangle-\langle\eta,x\rangle.
\]
The corresponding symplectic group is denoted by $\operatorname*{Sp}(n)$.

\section{Preliminary Material: Review\label{sec1}}

\subsection{The exponential of a linear form}

To make the definitions above rigorous, we have to give a precise sense to the
exponential operator $e^{i(\langle x^{\prime},x\rangle+\langle\xi^{\prime
},D\rangle)}$ and its variants. This can be done without any recourse to
operator functional calculus. Consider the Schr\"{o}dinger equation%
\begin{equation}
-i\frac{\partial v}{\partial t}=(\langle x_{0},x\rangle+\langle\xi
_{0},D\rangle)v \label{schr1}%
\end{equation}
with initial datum $u_{0}=v(\cdot,0)$ in $\mathcal{S}(\mathbb{R}^{n})$. Its
time-one solution $u_{1}=v(\cdot,1)$ is given by%
\[
u_{1}(x)=e^{i(\langle x_{0},x\rangle+\frac{1}{2}\langle x_{0},\xi_{0}\rangle
)}u_{0}(x+\xi_{0});
\]
writing formally the solution $u$ of (\ref{schr1}) as $e^{i(\langle
x_{0},x\rangle+\langle\xi_{0},D\rangle)t}u_{0}$ justifies the notation
\[
e^{i(\langle x_{0},x\rangle+\langle\xi_{0},D\rangle)}u_{0}=e^{i(\langle
x_{0},x\rangle+\frac{1}{2}\langle x_{0},\xi_{0}\rangle)}u_{0}(x+\xi_{0}).
\]
We will write $M(x_{0},\xi_{0})=e^{i(\langle x_{0},x\rangle+\langle\xi
_{0},D\rangle)}$; thus:%
\begin{equation}
M(x_{0},\xi_{0})u(x)=e^{i(\langle x_{0},x\rangle+\frac{1}{2}\langle x_{0}%
,\xi_{0}\rangle)}u(x+\xi_{0}). \label{M1}%
\end{equation}
Setting $x_{0}=0$ we have in particular $e^{i\langle\xi_{0},D\rangle
}u(x)=u(x+\xi_{0})$ for $u\in\mathcal{S}(\mathbb{R}^{n})$. One verifies by a
direct calculation using the definitions above that the \textquotedblleft
Baker--Campbell--Hausdorff formulas\textquotedblright\
\begin{equation}
M(x_{0},\xi_{0})=e^{-\frac{i}{2}\langle\xi_{0},x_{0}\rangle}e^{i\langle\xi
_{0},D\rangle}e^{i\langle x_{0},x\rangle}=e^{\frac{i}{2}\langle x_{0},\xi
_{0}\rangle}e^{i\langle x_{0},x\rangle}e^{i\langle\xi_{0},D\rangle}
\label{BKH}%
\end{equation}
hold. Notice that the operator $M(\xi_{0},-x_{0})$ is the Heisenberg operator
\[
T(x_{0},\xi_{0})=e^{i\sigma(x_{0},\xi_{0};x,D)}%
\]
\cite{Folland,108ter}, that is, $T(tx_{0},t\xi_{0})$ is the propagator of the
Schr\"{o}dinger equation $-i\partial_{t}v=\sigma(x_{0},\xi_{0};x,D)v$. The
operators $M(x_{0},\xi_{0})$ and $T(x_{0},\xi_{0})$ extend to continuous
operators $\mathcal{S}^{\prime}(\mathbb{R}^{n})\longrightarrow\mathcal{S}%
^{\prime}(\mathbb{R}^{n})$ whose restrictions to $L^{2}(\mathbb{R}^{n})$ are
unitary, and we have $M(x_{0},\xi_{0})^{\ast}=M(-x_{0},-\xi_{0})$,
$T(x_{0},\xi_{0})^{\ast}=T(-x_{0},-\xi_{0})$.

The Grossmann--Royer \cite{63,94} reflection operator $R(x_{0},\xi_{0})$ is
defined by%
\begin{equation}
R(x_{0},\xi_{0})=T(x_{0},\xi_{0})R(0,0)T(x_{0},\xi_{0})^{\ast} \label{GR1}%
\end{equation}
where $R(0,0)u(x)=u(-x)$. Explicitly, it is the unitary operator given by
\begin{equation}
R(x_{0},\xi_{0})u(x)=e^{2i\langle\xi_{0},x-x_{0}\rangle}u(2x_{0}-x).
\label{GR2}%
\end{equation}

\subsection{Weyl operators}

Let $A$ be a continuous linear operator $\mathcal{S}(\mathbb{R}^{n}%
)\longrightarrow\mathcal{S}^{\prime}(\mathbb{R}^{n})$. Using Schwartz's kernel
theorem (\cite{ho83}, Theorem 5.2.1) one shows that there exists a
distribution $K\in\mathcal{S}^{\prime}(\mathbb{R}^{n}\times\mathbb{R}^{n})$
such that $\langle Au,v\rangle=\langle K,\overline{v}\otimes u\rangle$ for all
$u,v\in\mathcal{S}(\mathbb{R}^{n})$; turning to integral notation, the
operator $A$ is thus formally given by%
\begin{equation}
Au(x)=\int K(x,y)u(y)dy. \label{aux}%
\end{equation}
The \textit{Weyl symbol} $a$ of $A$ is the tempered distribution on
$\mathbb{R}^{2n}$ given by the Fourier transform
\begin{equation}
a(x,\xi)=(2\pi)^{n/2}\int e^{-i\langle y,\xi\rangle}K(x+\tfrac{1}{2}%
y,x-\tfrac{1}{2}y)\,dy; \label{akern}%
\end{equation}
the action of the operator $A=\operatorname*{Op}_{\mathrm{W}}(a)$ on
$u\in\mathcal{S}(\mathbb{R}^{n})$ is thus given by the formula
\begin{equation}
Au(x)=(2\pi)^{-n}\int e^{i\langle x-y,\xi\rangle}a(\tfrac{1}{2}(x+y),\xi
)u(y)\,dyd\xi\label{Weyl2}%
\end{equation}
(the integral being interpreted in the distributional sense for $a\in
\mathcal{S}^{\prime}(\mathbb{R}^{2n})$). Performing the change of variables
$(y,\xi)\longmapsto(2x^{\prime}-x,\xi^{\prime})$ formula (\ref{Weyl2}) can be
rewritten in terms of the Grossmann--Royer operator (\ref{GR2}) as
\begin{equation}
Au=\pi^{-n}\int a(x^{\prime},\xi^{\prime})R(x^{\prime},\xi^{\prime})u\,dyd\xi.
\label{Weyl3}%
\end{equation}
Finally, applying the Parseval formula to the integral in (\ref{Weyl3}), we
get%
\begin{equation}
Au(x)=(2\pi)^{-n}\int\widehat{a}(x^{\prime},\xi^{\prime})M(x^{\prime}%
,\xi^{\prime})u(x)dyd\xi\label{Weyl4}%
\end{equation}
(see \cite{Springer}, \S 6.3.2).

\subsection{Shubin's $\tau$-operators}

Let $a\in\mathcal{S}(\mathbb{R}^{2n})$ and $\tau\in\mathbb{R}$; replacing the
definition (\ref{akern}) of the Weyl symbol with
\begin{equation}
a_{\tau}(x,\xi)=(2\pi)^{n/2}\int e^{-i\langle y,\xi\rangle}K(x+\tau
y,x-(1-\tau)y)\,dy \label{atau2}%
\end{equation}
we get the $\tau$-pseudodifferential operator (Shubin \cite{sh78}) $A_{\tau
}=\operatorname*{Op}_{\tau}(a)$:
\begin{equation}
A_{\tau}u(x)=(2\pi)^{-n}\int e^{i\langle x-y,\xi\rangle}a((1-\tau)x+\tau
y,\xi)u(y)\,dyd\xi; \label{atau1}%
\end{equation}
the case $\tau=\frac{1}{2}$ yields the Weyl operator $A=\operatorname*{Op}%
_{\mathrm{W}}(a)$. Equivalently, $A_{\tau}$ is the operator with Schwartz
kernel%
\begin{equation}
K_{\tau}(x,y)=(2\pi)^{-n}\int e^{i\langle x-y ,\xi\rangle}a((1-\tau)x+\tau
y,\xi)\,d\xi. \label{katau}%
\end{equation}
The operator $A_{\tau}$ is a continuous linear mapping $\mathcal{S}%
(\mathbb{R}^{n})\longrightarrow\mathcal{S}^{\prime}(\mathbb{R}^{n})$;
conversely, for every $\tau\in\mathbb{R}$, every such operator $A$ is a Shubin
$\tau$-operator, the $\tau$-symbol $a_{\tau}$ of $A$ being the distribution on
$\mathbb{R}^{2n}$ defined by the Fourier transform where $K$ is the Schwartz
kernel of $A$. One shows (\cite{Springer}, \S 9.2.1 and 9.3.1) that, as in the
case of Weyl operators, the operator $A_{\tau}$ can be written
\begin{equation}
A_{\tau}u=\pi^{-n}\int a(x^{\prime},\xi^{\prime})R_{\tau}(x^{\prime}%
,\xi^{\prime})u\, dx^{\prime}d\xi^{\prime}%
\end{equation}
where $R_{\tau}$ is given, for $\tau\neq\frac{1}{2}$, by
\begin{gather}
R_{\tau}(x^{\prime},\xi^{\prime})u=(\Theta_{\tau}\ast R(x^{\prime},\xi
^{\prime}))u\label{rotau}\\
\Theta_{\tau}(x,\xi)=\frac{2^{n}}{|2\tau-1|^{n}}\exp\left(  \frac{2i}{2\tau
-1}\langle x,\xi\rangle\right)  \label{thetatau}%
\end{gather}
and $\Theta_{1/2}(x,\xi)=\delta_{(x,\xi)}$. The Fourier decomposition of
$A_{\tau}$ is then given by
\begin{equation}
A_{\tau}u(x)=(2\pi)^{-n}\int\widehat{a}(x^{\prime},\xi^{\prime})M_{\tau
}(x^{\prime},\xi^{\prime})u(x)\,dx^{\prime}d\xi^{\prime} \label{atau3}%
\end{equation}
where, by definition,%
\begin{equation}
M_{\tau}(x,\xi)=e^{\frac{i}{2}(2\tau-1)\langle x,\xi\rangle}M(x,\xi).
\label{mtau}%
\end{equation}

It is convenient for our purposes to introduce the Shubin symbol classes
$\Gamma_{\rho}^{m}(\mathbb{R}^{2n})$ (Shubin, \cite{sh78}, \S 23). By
definition $a\in\Gamma_{\rho}^{m}(\mathbb{R}^{2n})$ ($m\in\mathbb{R}$,
$0<\rho\leq1$) if $a\in C^{\infty}(\mathbb{R}^{2n})$ and if for every
$\alpha\in\mathbb{N}^{2n}$ there exists $C_{\alpha}\geq0$ such that
\[
|D_{(x,\xi)}^{\alpha}a(x,\xi)|\leq C_{\alpha}(1+|x|+|\xi|)^{m-\rho|\alpha|}.
\]
Every polynomial $a$ of degree $m$ in the variables $x_{1},\ldots,x_{n}%
,\xi_{1},\ldots,\xi_{n}$ belongs to $\Gamma_{1}^{m}(\mathbb{R}^{2n})$. Using
standard estimates it is easy to check that if $a\in\Gamma_{\rho}%
^{m}(\mathbb{R}^{2n})$ then $A_{\tau}=\operatorname*{Op}_{\tau}(a)$ maps
continuously $\mathcal{S}(\mathbb{R}^{n})$ into $\mathcal{S}(\mathbb{R}^{n})$.

\subsection{Born--Jordan operators}

For $a\in\mathcal{S}(\mathbb{R}^{2n})$ the Born--Jordan operator
$A_{\mathrm{BJ}}=\operatorname*{Op}_{\mathrm{BJ}}(a)$ is the operator with
kernel $K_{\mathrm{BJ}}=\int_{0}^{1}K_{\tau}d\tau$ where $K_{\tau}$ is given
by (\ref{katau}); equivalently $A_{\mathrm{BJ}}=\int_{0}^{1}A_{\tau}d\tau$
where $A_{\tau}=\operatorname*{Op}_{\tau}(a)$. Using formulas (\ref{atau3})
and (\ref{mtau}) it is straightforward to obtain the Fourier decomposition of
$A_{\mathrm{BJ}}$:
\begin{equation}
A_{\mathrm{BJ}}u(x)=(2\pi)^{-n}\int\widehat{a}(x^{\prime},\xi^{\prime
})M_{\mathrm{BJ}}(x^{\prime},\xi^{\prime})u(x)\,dx^{\prime}d\xi^{\prime}
\label{abj}%
\end{equation}
with%
\begin{equation}
M_{\mathrm{BJ}}(x,\xi)=\operatorname{sinc}(\tfrac{1}{2}\langle x,\xi
\rangle)M(x,\xi) \label{absinc}%
\end{equation}
where, as usual, $\operatorname{sinc}t=\sin(t)/t$.

The terminology \textquotedblleft Born--Jordan operator\textquotedblright%
\ comes from the following observation: choose $n=1$ and assume that
$a_{r,s}(x,\xi)=x^{r}\xi^{s}$ where $r$ and $s$ are positive integers. Then
one has (\cite{Springer}, \S 9.1.2)
\[
\operatorname*{Op}\nolimits_{\tau}(a_{r,s})=\sum_{k=0}^{r}\binom{r}{k}\tau
^{k}(1-\tau)^{r-k}x^{k}D^{s}x^{r-k}.
\]
Integrating both sides of this equality from $0$ to $1$ with respect to $\tau$
we get, using the properties of the beta function,%
\begin{equation}
\operatorname*{Op}\nolimits_{\mathrm{BJ}}(a_{r,s})=\frac{1}{r+1}\sum_{k=0}%
^{r}x^{k}D^{s}x^{r-k} \label{BJmono1}%
\end{equation}
which is Born and Jordan's \textquotedblleft quantization
rule\textquotedblright\ \cite{16}. The following remark is important: one
proves by induction that%

\begin{equation}
\lbrack x^{r+1},D_{x}^{s+1}]=(s+1)i\sum\limits_{j=0}^{r}x^{r-j}D^{s}x^{j}
\label{commut1}%
\end{equation}
hence formula (\ref{BJmono1}) can be rewritten%
\begin{equation}
\operatorname*{Op}\nolimits_{\mathrm{BJ}}(a_{r,s})=\frac{1}{i(r+1)(s+1)}%
[x^{r+1},D^{s+1}]. \label{fundamental}%
\end{equation}

\begin{remark}
\label{oss} This identity is remarkable because it shows that Born--Jordan
operators with polynomial symbols in the $x,\xi$ variables can be expressed as
a sum of commutators (see in the context the paper \cite{JCPain} by Pain) and
that Born-Jordan quantization enjoys \eqref{two} at least for monomial symbols
in dimension 1. In other terms, the operators $\operatorname*{Op}%
\nolimits_{\mathrm{BJ}}(a_{r,s})$ are uniquely determined by the quantization
of monomials depending only on $x$ or $\xi$. We refer to Theorem
\ref{Theorem2} and Remark \ref{osservazione} below for the general case of
distribution symbols in arbitrary dimension.
\end{remark}

An important observation is the following: the adjoint of $A_{\mathrm{BJ}%
}=\operatorname*{Op}\nolimits_{\mathrm{BJ}}(a)$ with respect to the
sesquilinear product
\[
(u|v)=\int u(x)\overline{v}(x)dx
\]
on $\mathcal{S}(\mathbb{R}^{n})$ is the operator $A_{\mathrm{BJ}}^{\ast
}=\operatorname*{Op}\nolimits_{\mathrm{BJ}}(\overline{a})$ (hence
$A_{\mathrm{BJ}}$ is self-adjoint when $a$ is real). This follows from the
fact that $A_{\tau}^{\ast}=\operatorname*{Op}_{1-\tau}(\overline{a})$ if
$A_{\tau}=\operatorname*{Op}_{\tau}(a)$ (see \cite{transam,Springer}).

While the linear mapping%
\[
\operatorname*{Op}\nolimits_{\mathrm{W}}:\mathcal{S}^{\prime}(\mathbb{R}%
^{2n})\longrightarrow\mathcal{L(S}(\mathbb{R}^{n}),\mathcal{S}^{\prime
}(\mathbb{R}^{n}))
\]
which to every symbol $a\in\mathcal{S}^{\prime}(\mathbb{R}^{2n})$ associates
the corresponding Weyl operator $A=\operatorname*{Op}\nolimits_{\mathrm{W}%
}(a)$ is a continuous isomorphism \cite{Folland,ho85,Stein,wong}, this is not
true of the mapping
\[
\operatorname*{Op}\nolimits_{\mathrm{BJ}}:\mathcal{S}^{\prime}(\mathbb{R}%
^{2n})\longrightarrow\mathcal{L(S}(\mathbb{R}^{n}),\mathcal{S}^{\prime
}(\mathbb{R}^{n}))
\]
because it is not injective. In fact, set $m(x,\xi)=e^{i(\langle
x_{0},x\rangle+\langle\xi_{0},\xi\rangle)}$; we have $\widehat{m}=(2\pi
)^{n}\delta_{(x_{0},\xi_{0})}$ and hence by \eqref{abj} and \eqref{absinc} we
obtain
\begin{align}
\operatorname*{Op}\nolimits_{\mathrm{BJ}}(m)  &  =\int\delta_{(x_{0},\xi_{0}%
)}\operatorname{sinc}(\tfrac{1}{2}\langle x^{\prime},\xi^{\prime}%
\rangle)M(x^{\prime},\xi^{\prime})\,dx^{\prime}d\xi^{\prime}\nonumber\\
&  =\operatorname{sinc}(\tfrac{1}{2}\langle x_{0},\xi_{0}\rangle)M(x_{0}%
,\xi_{0}). \label{aggiunta}%
\end{align}
We thus have $\operatorname*{Op}\nolimits_{\mathrm{BJ}}(m)=0$ for all
$(x_{0},\xi_{0})$ such that $\langle x_{0},\xi_{0}\rangle\neq0$ and $\langle
x_{0},\xi_{0}\rangle\in2\pi\mathbb{Z}$. While the surjectivity of
$\operatorname*{Op}\nolimits_{\mathrm{W}}$ and $\operatorname*{Op}_{\tau}$ is
obvious using Schwartz's kernel theorem, the proof of the surjectivity of
$\operatorname*{Op}\nolimits_{\mathrm{BJ}}$ is rather tricky. The difficulty
comes from the following observation: for every $A\in\mathcal{L(S}%
(\mathbb{R}^{n}),\mathcal{S}^{\prime}(\mathbb{R}^{n}))$ there exists
$a\in\mathcal{S}^{\prime}(\mathbb{R}^{2n})$ such that $A=\operatorname*{Op}%
\nolimits_{\mathrm{W}}(a)$ hence the mapping $\operatorname*{Op}%
\nolimits_{\mathrm{BJ}}$ is surjective if and only if we can find
$b\in\mathcal{S}^{\prime}(\mathbb{R}^{2n})$ such that $\operatorname*{Op}%
\nolimits_{\mathrm{BJ}}(b)=\operatorname*{Op}\nolimits_{\mathrm{W}}(a)$.
Comparison of formulas (\ref{Weyl4}) and (\ref{abj}) shows that $b$ must be a
solution of the equation%
\[
\widehat{b}(x,\xi)\operatorname{sinc}(\tfrac{1}{2}\langle x,\xi\rangle
)=\widehat{a}(x,\xi);
\]
the determination of $\widehat{b}$ (and hence of $b$) thus requires a division
by the function $(x,\xi)\longmapsto\operatorname{sinc}(\tfrac{1}{2}\langle
x,\xi\rangle)$, which has infinitely many zeroes. We have proven in a recent
work \cite{28bis} with E. Cordero that the solution $\widehat{b}$ actually
exists in $\mathcal{S}^{\prime}(\mathbb{R}^{2n})$, but the method is quite
tricky and does not allow an explicit expression of $b$, neither does it allow
to produce any qualitative results about the regularity properties of $b$ in
terms of those of $a$. However, as we have shown in \cite{cogoni}, the
situation is much more satisfactory when one supposes that the symbol $a$
belongs to one of the Shubin symbol classes $\Gamma_{\rho}^{m}(\mathbb{R}%
^{2n})$. One has in this case the following result, which in a sense
trivializes Born--Jordan operators:

\begin{proposition}
If $\widehat{A}_{\mathrm{BJ}}=\operatorname*{Op}_{\mathrm{BJ}}(a)$ with
$a\in\Gamma_{\rho}^{m}(\mathbb{R}^{2n})$ there exists, for every $\tau
\in\mathbb{R}$, a symbol $a_{\tau}$ belonging to the same symbol class
$\Gamma_{\rho}^{m}(\mathbb{R}^{2n})$ such that $\widehat{A}_{\mathrm{BJ}%
}=\operatorname*{Op}_{\tau}(a_{\tau})$.

Conversely, for any given symbol $a_{\tau}\in\Gamma_{\rho}^{m}(\mathbb{R}%
^{2n})$ there exists a symbol $a\in\Gamma_{\rho}^{m}(\mathbb{R}^{2n})$ such
that $\operatorname*{Op}_{\mathrm{BJ}}(a)=\operatorname*{Op}_{\tau}(a_{\tau
})+R$ where $R$ is an operator with integral kernel in $\mathcal{S}%
(\mathbb{R}^{2n})$.
\end{proposition}

In particular, taking $\tau=\frac{1}{2}$, the operator $\widehat
{A}_{\mathrm{BJ}}=\operatorname*{Op}_{\mathrm{BJ}}(a)$ is a Weyl
pseudodifferential operator with symbol in the same Shubin class as $a$.

\section{The Characteristic Property\label{sec2}}

In this section we state and prove the main results of this paper.

\subsection{Born--Jordan quantization turns Poisson bracket into commutators}

Let $a,b\in\mathcal{C}^{\infty}(\mathbb{R}^{2n})$ and $X_{a},X_{b}$ the
corresponding Hamiltonian vector fields: $i_{X_{a}}\sigma+da=0$ and $i_{X_{b}%
}\sigma+db=0$. By definition the Poisson bracket of $a$ and $b$ is
$\{a,b\}=i_{X_{b}}i_{X_{a}}\sigma$. In coordinates, $X_{a}=(\partial_{\xi
}a,-\partial_{x}a)$ and $X_{b}=(\partial_{\xi}b,-\partial_{x}b)$, and%
\[
\{a,b\}=\sum_{|\alpha|=1}\partial_{x}^{\alpha}a\,\partial_{\xi}^{\alpha
}b-\partial_{x}^{\alpha}b\,\partial_{\xi}^{\alpha}a.
\]

Let us now define a convenient class of functions in $\mathbb{R}^{n}$.

\begin{definition}
\label{def2} Let $\mathcal{A}(\mathbb{R}^{n})$ be the space of all smooth
functions $f$ on $\mathbb{R}^{n}$ such that for every $\alpha\in\mathbb{N}%
^{n}$,
\[
|\partial^{\alpha}f(x)|\leq C_{\alpha}(1+|x|)^{m_{\alpha}}\quad x\in
\mathbb{R}^{n}
\]
for some constants $C_{\alpha}>0$ and $m_{\alpha}$ depending on $\alpha$.
\end{definition}

The relevance of this class of functions is that if $f\in\mathcal{A}%
(\mathbb{R}^{n})$ and $u\in\mathcal{S}(\mathbb{R}^{n})$ then $fu\in
\mathcal{S}(\mathbb{R}^{n})$.

\begin{theorem}
\label{Theorem1}Let $f\in\mathcal{A}(\mathbb{R}^{n})$ and $g\in\mathcal{A}%
(\mathbb{R}^{n})$; set $a=f\otimes1$ and $b=1\otimes g$. Then the operators
$\operatorname*{Op}\nolimits_{\mathrm{BJ}}(a)$, $\operatorname*{Op}%
\nolimits_{\mathrm{BJ}}(b)$ belong to $\mathcal{L(S}(\mathbb{R}^{n}%
),\mathcal{S}(\mathbb{R}^{n}))$ and we have
\begin{equation}
\lbrack\operatorname*{Op}\nolimits_{\mathrm{BJ}}(a),\operatorname*{Op}%
\nolimits_{\mathrm{BJ}}(b)]=i\operatorname*{Op}\nolimits_{\mathrm{BJ}%
}(\{a,b\}). \label{condbj}%
\end{equation}

\end{theorem}

\begin{proof}
An elementary calculation shows that
\begin{equation}
\operatorname*{Op}\nolimits_{\mathrm{BJ}}(a)u=fu,\quad\operatorname*{Op}%
\nolimits_{\mathrm{BJ}}(b)u=\mathcal{F}^{-1}(g\mathcal{F}u)=(2\pi
)^{-n/2}\mathcal{F}^{-1}g\ast u. \label{opabu}%
\end{equation}
Since $f\in\mathcal{A}(\mathbb{R}^{n})$ the mapping $u\longmapsto fu$ is
continuous on $\mathcal{S}(\mathbb{R}^{n})$, so that $\operatorname*{Op}%
\nolimits_{\mathrm{BJ}}(a)\in\mathcal{L(S}(\mathbb{R}^{n}),\mathcal{S}%
(\mathbb{R}^{n}))$; similarly since $g\in\mathcal{A}(\mathbb{R}^{n})$ the map
$u\longmapsto\mathcal{F}^{-1}g\ast u\in\mathcal{S}(\mathbb{R}^{n})$ is
continuous on $\mathcal{S}(\mathbb{R}^{n})$ and $\operatorname*{Op}%
\nolimits_{\mathrm{BJ}}(b)\in\mathcal{L(S}(\mathbb{R}^{n}),\mathcal{S}%
(\mathbb{R}^{n}))$ as well.

Now, we have
\[
\operatorname*{Op}\nolimits_{\mathrm{BJ}}(\{a,b\})=\operatorname*{Op}\left(
{\textstyle\sum_{|\alpha|=1}}
\partial_{x}^{\alpha}f\otimes\partial_{\xi}^{\alpha}g\right)  .
\]
The composed operators $\operatorname*{Op}\nolimits_{\mathrm{BJ}%
}(a)\operatorname*{Op}\nolimits_{\mathrm{BJ}}(b)$ and $\operatorname*{Op}%
\nolimits_{\mathrm{BJ}}(b)\operatorname*{Op}\nolimits_{\mathrm{BJ}}(a)$ are
thus well defined, belong to $\mathcal{L(S}(\mathbb{R}^{n}),\mathcal{S}%
(\mathbb{R}^{n}))$, and are given by%
\begin{align*}
\operatorname*{Op}\nolimits_{\mathrm{BJ}}(a)\operatorname*{Op}%
\nolimits_{\mathrm{BJ}}(b)(u)  &  =(2\pi)^{-n/2}(\mathcal{F}^{-1}g\ast u)f\\
\operatorname*{Op}\nolimits_{\mathrm{BJ}}(b)\operatorname*{Op}%
\nolimits_{\mathrm{BJ}}(a)(u)  &  =(2\pi)^{-n/2}\mathcal{F}^{-1}g\ast(fu).
\end{align*}
It follows that
\begin{equation}
\lbrack\operatorname*{Op}\nolimits_{\mathrm{BJ}}(a),\operatorname*{Op}%
\nolimits_{\mathrm{BJ}}(b)]u=(2\pi)^{-n/2}\left[  (\mathcal{F}^{-1}g\ast
u)f-\mathcal{F}^{-1}g\ast(fu)\right]  . \label{opab2}%
\end{equation}
Let us now show property (\ref{condbj}), that is
\begin{equation}
\lbrack\operatorname*{Op}\nolimits_{\mathrm{BJ}}(a),\operatorname*{Op}%
\nolimits_{\mathrm{BJ}}(b)]u=i\operatorname*{Op}\left(
{\textstyle\sum_{|\alpha|=1}}
\partial_{x}^{\alpha}f\otimes\partial_{\xi}^{\alpha}g\right)  u \label{opab1}%
\end{equation}
for all $u\in\mathcal{S}(\mathbb{R}^{n})$.

Since $\widehat{\partial_{x_{j}}f}(\xi)=i\xi_{j}\widehat{f}(\xi)$, by
\eqref{abj} and \eqref{absinc} we have%
\begin{equation}
\operatorname*{Op}\nolimits_{\mathrm{BJ}}\left(
{\textstyle\sum_{|\alpha|=1}}
\partial_{x}^{\alpha}f\otimes\partial_{\xi}^{\alpha}g\right)  u(x)=(-2)(2\pi
)^{-n}I(x) \label{35bis}%
\end{equation}
where%
\[
I(x)=\int\widehat{f}(x^{\prime})\widehat{g}(\xi^{\prime})\sin(\tfrac{1}%
{2}\langle x^{\prime},\xi^{\prime}\rangle)e^{i(\langle x^{\prime}%
,x\rangle+\frac{1}{2}\langle x^{\prime},\xi^{\prime}\rangle)}u(x+\xi^{\prime
})\,dx^{\prime}d\xi^{\prime}.
\]
Writing $\sin t=(e^{it}-e^{-it})/2i$ we have $I(x)=I_{1}(x)+I_{2}(x)$ where%
\begin{align*}
I_{1}(x)  &  =\frac{1}{2i}\int\widehat{f}(x^{\prime})\widehat{g}(\xi^{\prime
})e^{i(\langle x^{\prime},x\rangle+\langle x^{\prime},\xi^{\prime}\rangle
)}u(x+\xi^{\prime})dx^{\prime}d\xi^{\prime}\\
I_{2}(x)  &  =-\frac{1}{2i}\int\widehat{f}(x^{\prime})\widehat{g}(\xi^{\prime
})e^{i\langle x^{\prime},x\rangle}u(x+\xi^{\prime})dx^{\prime}d\xi^{\prime}.
\end{align*}
Performing the change of variables $(x^{\prime},\xi^{\prime})\longmapsto
(x^{\prime\prime},\xi^{\prime\prime}-x)$ these integrals become%
\begin{align*}
I_{1}(x)  &  =\frac{1}{2i}\int\widehat{f}(x^{\prime\prime})\widehat{g}%
(\xi^{\prime\prime}-x)e^{i\langle x^{\prime\prime},\xi^{\prime\prime}\rangle
}u(\xi^{\prime\prime})\,dx^{\prime\prime}d\xi^{\prime\prime}\\
I_{2}(x)  &  =-\frac{1}{2i}\int\widehat{f}(x^{\prime\prime})\widehat{g}%
(\xi^{\prime\prime}-x^{\prime})e^{i\langle x^{\prime\prime},x\rangle}%
u(\xi^{\prime\prime})\,dx^{\prime\prime}d\xi^{\prime\prime}.
\end{align*}
Using successively the identity $\widehat{g}(\xi^{\prime\prime}-x^{\prime
})=\mathcal{F}^{-1}g(x-\xi^{\prime\prime})$, Fubini's theorem, and the Fourier
inversion formula we get the expressions%
\begin{align*}
I_{1}(x)  &  =\frac{(2\pi)^{n/2}}{2i}\left[  \mathcal{F}^{-1}g\ast(fu)\right]
(x)\\
I_{2}(x)  &  =-\frac{(2\pi)^{n/2}}{2i}f(x)[\mathcal{F}^{-1}g\ast u](x)
\end{align*}
and hence, by \eqref{35bis},%
\[
\operatorname*{Op}\nolimits_{\mathrm{BJ}}\left(
{\textstyle\sum_{|\alpha|=1}}
\partial_{x}^{\alpha}f\otimes\partial_{\xi}^{\alpha}g\right)  u=-i(2\pi
)^{-n/2}\left[  f(\mathcal{F}^{-1}g\ast u)-\mathcal{F}^{-1}g\ast(fu)\right]
.
\]
Together with \eqref{opab2} this proves the equality (\ref{opab1}).
\end{proof}

\medskip Let us call $h\in\mathcal{S}^{\prime}(\mathbb{R}^{2n})$ a
\textquotedblleft physical Hamiltonian\textquotedblright\ if $h(x,\xi
)=f(x)+g(\xi)$ with $f,g\in\mathcal{A}(\mathbb{R}^{n})$. The following
consequence of Theorem \ref{Theorem1} is straightforward:

\begin{corollary}
\label{corollary}Let $h$ and $k$ be physical Hamiltonians, and set
$H=\operatorname*{Op}\nolimits_{\mathrm{BJ}}(h)$, $K=\operatorname*{Op}%
\nolimits_{\mathrm{BJ}}(k)$. We have
\begin{equation}
\lbrack H,K]=i\operatorname*{Op}\nolimits_{\mathrm{BJ}}(\{h,k\}).
\label{condphys}%
\end{equation}

\end{corollary}

\begin{proof}
Writing $h(x,\xi)=f(x)+g(\xi)$ and $k(x,\xi)=d(x)+e(\xi)$ we have, using the
linearity of the Poisson bracket and the fact that $\{f,d\}=\{g,e\}=0$,%
\[
\{h,k\}=\{f,e\}+\{g,d\}.
\]
Let $F=\operatorname*{Op}\nolimits_{\mathrm{BJ}}(f)$, and so on. In view of
the equalities (\ref{opabu}) we have $FD=DF$ and $GE=EG$ and hence
\[
\lbrack H,K]=[F,E]+[G,D]
\]
formula (\ref{condphys}) follows since $[F,E]=i\operatorname*{Op}%
\nolimits_{\mathrm{BJ}}(\{f,e\})$ and $[G,D]=i\operatorname*{Op}%
\nolimits_{\mathrm{BJ}}(\{g,d\})$ in view of (\ref{condbj}).
\end{proof}

\subsection{The characteristic property of Born-Jordan operators}

We now prove a converse of Theorem \ref{Theorem1}.

Consider the space
\[
\mathcal{A}_{0}(\mathbb{R}^{n})=\{e^{i\langle x^{\prime},\cdot\rangle
}:\ x^{\prime}\in\mathbb{R}^{n}\}
\]
of purely imaginary exponentials in $\mathbb{R}^{n}$.

\begin{theorem}
\label{Theorem2} Let $\operatorname*{Op}:\mathcal{S}^{\prime}(\mathbb{R}%
^{2n})\longrightarrow\mathcal{L(S}(\mathbb{R}^{n}),\mathcal{S}^{\prime
}(\mathbb{R}^{n}))$ be such that%
\begin{equation}
\operatorname*{Op}(f\otimes1)u=fu,\quad\operatorname*{Op}(1\otimes
f)u=\mathcal{F}^{-1}(f\mathcal{F}u) \label{opone}%
\end{equation}
if $f\in\mathcal{A}_{0}(\mathbb{R}^{n})$, and
\begin{equation}
\lbrack\operatorname*{Op}(a),\operatorname*{Op}(b)]=i\operatorname*{Op}%
(\{a,b\}) \label{condop}%
\end{equation}
for all $a=f\otimes1$ with $f\in\mathcal{A}_{0}(\mathbb{R}^{n})$ and
$b=1\otimes g$ with $g\in\mathcal{A}_{0}(\mathbb{R}^{n})$.

Then
\[
\operatorname*{Op}(a)=\operatorname*{Op}\nolimits_{\mathrm{BJ}}(a)
\]
for all $a\in\mathcal{S}^{\prime}(\mathbb{R}^{2n})$.
\end{theorem}

We need the following known density result (we report on the short proof for
the sake of completeness).

\begin{lemma}
\label{Lemma1} The linear span of $\mathcal{A}_{0}(\mathbb{R}^{n})$ is dense
in $\mathcal{S}^{\prime}(\mathbb{R}^{n})$.
\end{lemma}

\begin{proof}
Since the Fourier transform is an isomorphism of $\mathcal{S}^{\prime
}(\mathbb{R}^{n})$, it is sufficient to prove that the linear span of the set
of Dirac delta functions $\delta_{x^{\prime}}$, $x^{\prime}\in\mathbb{R}^{n}$,
is dense in $\mathcal{S}^{\prime}(\mathbb{R}^{n})$. To this end, observe that
$C^{\infty}_{c}(\mathbb{R}^{n})$ is dense in $\mathcal{S}^{\prime}%
(\mathbb{R}^{n})$; on the other hand, every function $f\in C^{\infty}%
_{c}(\mathbb{R}^{n})$ is the limit in $\mathcal{S}^{\prime}(\mathbb{R}^{n})$
of the finite sums $(1/k)^{n}\sum_{x^{\prime}\in(1/k)\mathbb{Z}^{n}}
f(x^{\prime})\delta_{x^{\prime}}$ as $k\to+\infty$, as one sees by
approximating the pairing $\langle\langle f,\phi\rangle\rangle=\int
f(x)\phi(x)\, dx$, $\phi\in\mathcal{S}(\mathbb{R}^{n})$, by Riemann sums.
\end{proof}

\medskip

\begin{proof}
[Proof of Theorem \ref{Theorem2}]By Lemma \ref{Lemma1} (applied in dimension
$2n$) the exponentials $e^{i(\langle x^{\prime},\cdot\rangle+\langle
\xi^{\prime},\cdot\rangle)}$, $x^{\prime},\xi^{\prime}\in\mathbb{R}^{n}$, span
a dense subspace of $\mathcal{S}^{\prime}(\mathbb{R}^{2n})$. Since both the
quantizations $\operatorname*{Op}$ and $\operatorname*{Op}_{\mathrm{BJ}}$ are
linear and continuous on $\mathcal{S}^{\prime}(\mathbb{R}^{2n})$, it is
sufficient to prove that they coincide on $\mathcal{A}_{0}(\mathbb{R}^{2n})$.
By \eqref{aggiunta} this amounts to prove that
\begin{equation}
\operatorname*{Op}(e^{i(\langle x^{\prime},\cdot\rangle+\langle\xi^{\prime
},\cdot\rangle)})=\operatorname{sinc}(\tfrac{1}{2}\langle x^{\prime}%
,\xi^{\prime}\rangle)e^{i(\langle x^{\prime},x\rangle+\langle\xi^{\prime
},D\rangle)}. \label{eq2}%
\end{equation}
By \eqref{opone} we have%
\begin{equation}
\operatorname*{Op}(e^{i\langle x^{\prime},\cdot\rangle}\otimes
1)u(x)=e^{i\langle x^{\prime},x\rangle}u(x),\quad\operatorname*{Op}(1\otimes
e^{i\langle\xi^{\prime},\cdot\rangle})u(x)=e^{i\langle\xi^{\prime},D\rangle
}u(x) \label{opop}%
\end{equation}
for $u\in\mathcal{S}(\mathbb{R}^{n})$.

Assume now $\langle x^{\prime},\xi^{\prime}\rangle\neq0$. The condition
(\ref{condop}) implies%
\[
\lbrack\operatorname*{Op}(e^{i\langle x^{\prime},\cdot\rangle}\otimes
1),\operatorname*{Op}(1\otimes e^{i\langle\xi^{\prime},\cdot\rangle}%
)]=\frac{1}{i}\langle x^{\prime},\xi^{\prime}\rangle\operatorname*{Op}%
(e^{i(\langle x^{\prime},\cdot\rangle+\langle\xi^{\prime},\cdot\rangle)}),
\]
that is
\[
\operatorname*{Op}(e^{i(\langle x^{\prime},\cdot\rangle+\langle\xi^{\prime
},\cdot\rangle)})=\frac{i}{\langle x^{\prime},\xi^{\prime}\rangle
}[\operatorname*{Op}(e^{i\langle x^{\prime},\cdot\rangle}\otimes
1),\operatorname*{Op}(1\otimes e^{i\langle\xi^{\prime},\cdot\rangle})].
\]
In view of \eqref{opop} and (\ref{BKH}) we have%
\begin{align*}
\operatorname*{Op}(e^{i\langle x^{\prime},\cdot\rangle}\otimes
1)\operatorname*{Op}(1\otimes e^{i\langle\xi^{\prime},\cdot\rangle})  &
=e^{-\frac{1}{2i}\langle x^{\prime},\xi^{\prime}\rangle}e^{i(\langle
x^{\prime},x\rangle+\langle\xi^{\prime},D\rangle)}\\
\operatorname*{Op}( 1\otimes e^{i\langle\xi^{\prime},\cdot\rangle
})\operatorname*{Op}( e^{i\langle x^{\prime},\cdot\rangle}\otimes1)  &
=e^{\frac{1}{2i}\langle x^{\prime},\xi^{\prime}\rangle}e^{i(\langle x^{\prime
},x\rangle+\langle\xi^{\prime},D\rangle)}%
\end{align*}
and hence%
\[
\operatorname*{Op}(e^{i(\langle x^{\prime},\cdot\rangle+\langle\xi^{\prime
},\cdot\rangle)})=\frac{i}{\langle x^{\prime},\xi^{\prime}\rangle}\left(
e^{-\frac{1}{2i}\langle x^{\prime},\xi^{\prime}\rangle}-e^{\frac{1}{2i}\langle
x^{\prime},\xi^{\prime}\rangle}\right)  e^{i(\langle x^{\prime},x\rangle
+\langle\xi^{\prime},D\rangle)}%
\]
which is (\ref{eq2}).

The case $\langle x^{\prime},\xi^{\prime}\rangle=0$ follows by continuity,
because both sides of (\ref{eq2}) are continuous functions of $x^{\prime}%
,\xi^{\prime}$ valued in $\mathcal{L}(\mathcal{S}(\mathbb{R}^{n}%
),\mathcal{S}^{\prime}(\mathbb{R}^{n}))$.

This concludes the proof.
\end{proof}

\begin{remark}
\label{osservazione} As an alternative, since we already proved in Theorem
\ref{Theorem1} that the Born-Jordan quantization enjoys the properties in the
statement of Theorem \ref{Theorem2}, in the proof of the latter we could limit
ourselves to showing that there is at most one quantization $\mathrm{Op}%
:\mathcal{S}^{\prime}(\mathbb{R}^{2n})\longrightarrow\mathcal{L(S}%
(\mathbb{R}^{n}),\mathcal{S}^{\prime}(\mathbb{R}^{n}))$ satisfying those
properties. Now, those conditions force the values $\mathrm{Op}(a)$ when $a$
is a symbol of the type
\begin{equation}
e^{i(\langle x^{\prime},\cdot\rangle+\langle\xi^{\prime},\cdot\rangle)}%
,\quad\langle x^{\prime},\xi^{\prime}\rangle\not =0,\label{oss1}%
\end{equation}
because, as we saw, such symbols can be written (up to a multiplicative
constant) as the Poisson bracket of symbols in $\mathcal{A}_{0}(\mathbb{R}%
^{n})$. On the other hand, symbols of the type \eqref{oss1} span a dense
subset of $\mathcal{S}^{\prime}(\mathbb{R}^{2n})$ (this is true by Lemma
\ref{Lemma1} without the condition \textquotedblleft$\langle x^{\prime}%
,\xi^{\prime}\rangle\not =0$\textquotedblright, but also under this additional
condition, because those exponential functions are continuous, as functions of
$x^{\prime},\xi^{\prime}$, valued in $\mathcal{S}^{\prime}(\mathbb{R}^{2n})$
and the set of $(x',\xi')\in\mathbb{R}^{2n}$ such that $\langle x^{\prime},\xi^{\prime
}\rangle\not =0$ is dense in $\mathbb{R}^{2n}$.
\end{remark}

\section{Discussion\label{secdisc}}

The original concept of quantization in physics consists in trying to assign
to \textquotedblleft observables\textquotedblright\ (= real valued symbols) on
$\mathbb{R}^{2n}$ self-adjoint operators on a Hilbert space (usually
$L^{2}(\mathbb{R}^{2n})$) according to certain rules, dictated by physical
considerations. Mathematically speaking, this amounts in constructing a
continuous mapping $\operatorname*{Op}$ from some Poisson algebra of functions
defined on $\mathbb{R}^{2n}$ and such that:

\medskip(i) The operators $\operatorname*{Op}(x_{j})$ and $\operatorname*{Op}%
(\xi_{j})$ are given by $\operatorname*{Op}(x_{j})u=x_{j}u$ and
$\operatorname*{Op}(\xi_{j})u=D_{j}u$;

\medskip(ii) $[\operatorname*{Op}(a),\operatorname*{Op}%
(b)]=i\operatorname*{Op}(\{a,b\})$ (when $[\operatorname*{Op}%
(a),\operatorname*{Op}(b)]$ exists).

\medskip\noindent These rules are often complemented by other conditions, for instance, the
\textquotedblleft von Neumann\textquotedblright\ rule

\medskip(iii) $\operatorname*{Op}(\phi\circ a)=\phi(\operatorname*{Op}(a))$
where $\phi$ is a real function for which $\phi(\operatorname*{Op}(a))$ is defined.

\medskip Suppose now that the symbols $a$ and $b$ are quadratic polynomials:
$a(x,\xi)=\frac{1}{2}\langle M_{a}z,z\rangle$, $b(x,\xi)=\frac{1}{2}\langle
M_{b}z,z\rangle$, $z=(x,\xi)$, $M_{a}$ and $M_{b}$ being symmetric matrices.
The flows of the corresponding Hamiltonian vector fields $X_{a}$ and $X_{b}$
are linear hence consist of one-parameter subgroups $(S_{a}^{t})_{t\in
\mathbb{R}}$ and $(S_{b}^{t})_{t\in\mathbb{R}}$ of the symplectic group
$\operatorname*{Sp}(n)$. Using the path-lifting theorem it follows that we can
lift, in a unique way $(S_{a}^{t})_{t\in\mathbb{R}}$ and $(S_{b}^{t}%
)_{t\in\mathbb{R}}$ to one-parameter subgroups of any of the covering groups
$\operatorname*{Sp}_{q}(n)$ of $\operatorname*{Sp}(n)$. Choosing $q=2$ and
identifying $\operatorname*{Sp}_{2}(n)$ with the metaplectic group
$\operatorname*{Mp}(n)$, we obtain two one-parameter subgroups $(\widehat
{S}_{a}^{t})_{t\in\mathbb{R}}$ and $(\widehat{S}_{b}^{t})_{t\in\mathbb{R}}$ of
unitary operator acting on $L^{2}(\mathbb{R}^{n})$. It now requires some
calculations to show that
\[
i\frac{d}{dt}\widehat{S}_{a}^{t}=A(x,D)\widehat{S}_{a}^{t}%
\]
where the symmetric operator $A(x,D)$ is formally given by
\[
A(x,D)=\frac{1}{2}\langle M_{a}(x,D),(x,D)\rangle=\operatorname*{Op}%
\nolimits_{\mathrm{W}}(a).
\]
(and likewise for $\widehat{S}_{b}^{t}$); a few more calculations \cite{GS,62}
then show we have%
\[
\lbrack\operatorname*{Op}\nolimits_{\mathrm{W}}(a),\operatorname*{Op}%
\nolimits_{\mathrm{W}}(b)]=i\operatorname*{Op}\nolimits_{\mathrm{W}}(\{a,b\})
\]
which is condition (ii). (The latter is easily extended to non-homogeneous
quadratic polynomials). Now, the essence of the Groenewold--Van Hove no-go
result is that there exists no quantization $\operatorname*{Op}$ whose
restriction to the Poisson algebra of quadratic polynomials coincides with
$\operatorname*{Op}\nolimits_{\mathrm{W}}$ and still satisfies (ii).

It suffices in fact to show that (ii) cannot hold for the Poisson algebra of
polynomials; the proof then boils down to the following observation. Consider
the monomial $x^{2}\xi^{2}$; it can be written in two different ways using
Poisson brackets, namely%
\begin{equation}
x^{2}\xi^{2}=\frac{1}{9}\{x^{3},\xi^{3}\}=\frac{1}{3}\{x^{2}\xi,x\xi^{2}\}.
\label{x2ksi2}%
\end{equation}
Let us assume that such an $\operatorname*{Op}$ exists; then we would have
\[
\frac{1}{9}\operatorname*{Op}(\{x^{3},\xi^{3}\})=\frac{1}{3}\operatorname*{Op}%
(\{x^{2}\xi,x\xi^{2}\})
\]
and hence%
\begin{equation}
\frac{1}{9i}[\operatorname*{Op}(x^{3}),\operatorname*{Op}(\xi^{3})]=\frac
{1}{3i}[\operatorname*{Op}(x^{2}\xi),\operatorname*{Op}(x\xi^{2})].
\label{disc1}%
\end{equation}
Assuming $\operatorname*{Op}(x^{3})$ is multiplication by $x^{3}$ and
$\operatorname*{Op}(\xi^{3})=D^{3}$ we get, after some calculations%
\begin{equation}
\frac{1}{9i}[\operatorname*{Op}(x^{3}),\operatorname*{Op}(\xi^{3})]=x^{2}%
D^{2}-2ixD-\frac{2}{3} \label{disc2}%
\end{equation}
and similarly, writing
\[
\mathrm{Op}(x^{2}\xi)=\frac{1}{6i}[\mathrm{Op}(x^{3}),\mathrm{Op}(\xi
^{2})],\quad\mathrm{Op}(x\xi^{2})=\frac{1}{6i}[\mathrm{Op}(x^{2}%
),\mathrm{Op}(\xi^{3})]
\]
we obtain
\begin{equation}
\frac{1}{3i}[\operatorname*{Op}(x^{2}\xi),\operatorname*{Op}(x\xi^{2}%
)]=x^{2}D^{2}-2ixD-\frac{1}{3} \label{disc3}%
\end{equation}
hence a contradiction. Now comes the crucial point: the conflict between
(\ref{disc2}) and (\ref{disc3}) disappears when one chooses
$\operatorname*{Op}=\operatorname*{Op}_{\mathrm{BJ}}$ and one requires a
weaker form of (ii), namely

\medskip(ii')\quad $[\operatorname*{Op}(a),\operatorname*{Op}%
(b)]=i\operatorname*{Op}(\{a,b\})$ for all $a=f\oplus g$ and $b=h\oplus k$.

\medskip\noindent This weaker rule indeed excludes the condition (\ref{disc1})
and in fact we have%
\[
\operatorname*{Op}\nolimits_{\mathrm{BJ}}(x^{2}\xi^{2})=x^{2}D^{2}%
-2ixD-\frac{2}{3},
\]
which agrees with \eqref{disc2}.

Since $\operatorname*{Op}\nolimits_{\mathrm{BJ}}(a)=\operatorname*{Op}%
\nolimits_{\mathrm{WJ}}(a)$ for all quadratic symbols (see
\cite{transam,Springer,golu}), Born--Jordan quantization really is the most
natural quantization. In view of Corollary \ref{corollary} we moreover have
\begin{equation}
\lbrack H,K]=i\operatorname*{Op}\nolimits_{\mathrm{BJ}}(\{h,k\})
\end{equation}
for all symbols of the type $h(x,\xi)=f(x)+g(\xi)$ and $k(x,\xi)=d(x)+e(\xi)$;
from a physical point of view this means that the Dirac correspondence holds
for all Hamiltonians of the traditional type \textquotedblleft kinetic energy
plus potential\textquotedblright, and this result does not hold for any other
quantization; in particular if one replaces $\operatorname*{Op}%
\nolimits_{\mathrm{BJ}}$ with the Weyl correspondence $\operatorname*{Op}%
\nolimits_{\mathrm{W}}$ we have generally%
\begin{equation}
\lbrack H,K]\neq i\operatorname*{Op}\nolimits_{\mathrm{W}}(\{h,k\}).
\end{equation}

\medskip\noindent\textbf{Acknowledgement} \textit{The \ first author has been
supported by an Austrian Science Fund (FWF) Grant (project number P27773).}

\end{document}